\documentclass[12pt]{amsart}
\usepackage{amsfonts,color,morefloats,pslatex}
\usepackage{amssymb,amsthm, amsmath,latexsym}

\newtheorem{theorem}{Theorem}
\newtheorem{lemma}[theorem]{Lemma}

\newtheorem{example}{Example}

\def\qi#1 {\fbox {\footnote {\ }}\ \footnotetext { From Qi: {\color{red}#1}}}

\begin{document}
\title[Infinite families of $2$-designs from linear codes]{Infinite families of $2$-designs from a class of linear codes related to Dembowski-Ostrom functions}
\author{Rong Wang}
\address{College of Mathematics and Statistics, Northwest Normal University, Lanzhou, Gansu 730070, China}
\curraddr{}
\email{rongw113@126.com}
\author{Xiaoni Du}
\address{College of Mathematics and Statistics, Northwest Normal University, Lanzhou, Gansu 730070, China}
\curraddr{}
\email{ymldxn@126.com}
\author{Cuiling Fan}
\address{ School of Mathematics, Southwest Jiaotong University, Chengdu, Sichuan 610000, China}
\curraddr{}
\email{cuilingfan@163.com}
\author{Zhihua Niu}
\address{School of Computer Engineering and Science, Shanghai University, Shanghai, 200444, China}
\curraddr{}
\email{zhniu@staff.shu.edu.cn}

\thanks{}
\date{\today}

\maketitle

\begin{abstract}
Due to their important applications  to  
coding theory, cryptography, communications and statistics,  combinatorial $t$-designs have been attracted lots of research interest for decades.
The interplay between coding theory and $t$-designs has on going for many years. As we all known, $t$-designs can be used to derive linear codes over any finite field, as well as the supports of all codewords with a fixed weight in a code also may hold a $t$-design. In this paper, we first construct a class of linear codes from cyclic codes related to Dembowski-Ostrom functions. By using exponential sums, we then  determine the weight distribution of the linear codes. Finally, we obtain infinite families of $2$-designs from the supports of all codewords with a fixed weight in these codes. Furthermore, the parameters of $2$-designs are calculated explicitly.
\end{abstract}


\section{Introduction}
We start this article by introducing the relevant conception to $t$-designs. Let $v, k, \lambda$, and $t$ be positive integers such that $v>k\geq t.$ Let $\mathcal{P}$ be a set with cardinality $\mid\mathcal{P}\mid=v$ and $\mathcal{B}$ be a multi-set of  $k$-subsets of $\mathcal{P}$ with $\mid \mathcal{B} \mid =b$.  The elements of $\mathcal{P}$ are called as points and that of $\mathcal{B}$ as blocks. A $t$-$(v,k,\lambda)$ design is a pair $\mathbb{D}=(\mathcal{P}, \mathcal{B})$ such that each block contains exactly $k$ points, and every set of $t$ distinct points is contained in exactly $\lambda$ blocks. A $t$-$(v,k,\lambda)$ design without repeated blocks is called a {\em simple} $t$-design. A $t$-design is called {\em symmetric} if $v=b$ and {\em trivial} if $k=t$ or $k=v$. In the following, we pay only attention to simple $t$-designs with $t < k < v$. When $t\geq
2$,  we call the $t$-$(v, k, 1)$ design a {\em Steiner system}. 
Furthermore, the number of blocks in a $t$-design is
\begin{equation}\label{condition}
b=\frac{\lambda {v \choose t}}{ {k \choose t}}.
\end{equation}

The study of $t$-designs as a mathematical discipline has  been a very interesting subject for many years due to their wide applications in coding theory and analysis of statistical experiments. $t$-Designs have many other applications as well, such as cryptography, communications and other engineering areas.
It is well known that the interplay between linear codes and $t$-designs has been attracted a lot of attention for decades. A linear code over any {\em finite field} can be derived from the incidence matrix of a $t$-design and much progress has been made (see for example~\cite{AK92,DD15,Ton98,Ton07}). On the other hand, both linear and nonlinear codes can hold $t$-designs and some of $2$-designs and $3$-designs are derived from codes, we refer the readers to \cite{AK92,BJL99,CM06,Ding15b,KP95,KP03,MS771,RR10, Ton98,Ton07,TDX19} for examples. The main technical idea is based on the support design, that is, if one index the coordinates of a codeword in a code $\mathcal{C}$ by $(0, 1, \ldots, n-1)$ and let $\mathcal{P}=(0, 1, \ldots, n-1)$. And so the pair $(\mathcal{P},\mathcal{B}_i)$ may be a $t$-$(n,i,\lambda)$ design for some positive $\lambda$, where $\mathcal{B}_i$ is the set of the supports of all codewords with weight $i$ in $\mathcal{C}$ for each $i$ with
$A_i\neq 0$ \cite{DL17} (we will describe the definition  below).

There exist two classical approaches to obtain infinitely families of $t$-designs from linear codes. The first is to employ the Assmus-Mattson Theorem given in \cite{AM69, AM74} and we refer the readers to \cite{Ding182,DL17} for details. The second is to study the automorphism group of a linear code $\mathcal{C}$. That is, if the permutation part of the automorphism group acts $t$-transitively on a code $\mathcal{C}$, then the code $\mathcal{C}$ holds $t$-designs \cite{AK92,MS771}. Only a small amount of work in the direction has been done. For example, Ding {\em et al.} \cite{DTV19} and Du {\em et al.} \cite{DWTW190,DWTW191,DWF192} have derived infinite families of $3$-designs and $2$-designs from several different classes of affine-invariant codes, respectively. 
In this paper, we will obtain infinite families of $2$-designs by studying the automorphism group of a class of linear codes we construct.

The remainder of this paper is organized as follows. In Section \ref{code-construction}, we promote the definition of the linear codes we discuss, which are derived from cyclic codes related to a class of Dembowski-Ostrom functions.  In Section \ref{section-2}, we introduce some notation and preliminary results on affine invariant codes, $2$-designs, exponential sums and cyclotomic fields, which will be used in subsequent sections. In Section \ref{section-3}, we determine explicitly the weight distribution of the codes we defined by using certain exponential sums and obtain infinite families of $2$-designs and their parameters. Section \ref{section-4} proves the main results. Furthermore, we also 
use Magma programs to give some examples. And Section \ref{section-5} concludes the paper.

\section{Linear codes and our construction} \label{code-construction}

  Let $p$ be an odd prime and  $m$ any positive integer. Let $\mathbb{F}_q$ denote the finite field with $q=p^m$ elements and $\mathbb{F}^*_q=\mathbb{F}_q\backslash\{0\}$. An $[n,k,\delta]$  {\em linear code} $\mathcal{C}$ over $\mathbb{F}_p$
 is called  {\em cyclic code} if each codeword $(c_0,c_1,\ldots,c_{n-1})\in\mathcal{C}$ implies $(c_{n-1}, c_0, c_1, \ldots,
c_{n-2}) \in \mathcal{C}$. Any cyclic code $\mathcal{C}$ can be expressed as $\mathcal{C} = \langle g(x) \rangle$, where $g(x)$ is monic and has the least degree. The polynomial  $g(x)$ is called the  {\em generator polynomial} and $h(x)=(x^n-1)/g(x)$ is referred to as the {\em parity-check polynomial} of $\mathcal{C}$. The {\em dual} of a cyclic code $\mathcal{C}$, denoted by $\mathcal{C}^{\bot}$, is a cyclic code with length $n$, dimension $n-k$ and  generator polynomial $x^kh(x^{-1})/h(0)$.  The  {\em extended code} $\overline{\mathcal{C}}$ of $\mathcal{C}$  is defined by
$$
\overline{\mathcal{C}}=\{(c_0, c_1, \ldots, c_n) \in \mathbb{F}_p^{n+1}: (c_0, c_1, \ldots, c_{n-1}) \in \mathcal{C} ~\mathrm{with} ~\sum^n_{i=0}c_i=0\}.
$$

Let $A_i$ be the number of codewords with Hamming weight $i$ in a code $\mathcal{C}$. The {\em weight enumerator} of $\mathcal{C}$ is defined by $\sum_{i=0}^nA_iz^i,$
and the sequence $(1, A_1, \ldots, A_n)$ is called the {\em weight distribution} of the code $\mathcal{C}.$ The weight distribution is an important research object in coding
theory because it contains some crucial information of the code, for example,  the minimum weight, which
determines the error correction ability of the code. If the number of nonzero $A_i$'s with $1 \leq i \leq n$ is $w$, then we call $\mathcal{C}$ a {\em $w$-weight code}.

If  $\mathbf{c}=(c_0, c_1, \ldots, c_{n-1})$ is a codeword in a code $\mathcal{C},$ then  {\em support} of $\mathbf{c}$ is defined by
$$Suppt(\mathbf{c})=\{0\leq i \leq n-1: c_i\neq 0\}\subseteq \{0, 1, \ldots, n-1\}.$$

For any integer $0\leq j< n$, the {\em $p$-cyclotomic coset} of $j$ modulo $n$ is defined by
$$C_j=\{jp^i \pmod n: 0 \le i \le \ell_j-1\},$$
where $\ell_j$ is the smallest positive integer such that $j\equiv jp^{\ell_j}\pmod n.$ The smallest integer in $C_j$ is called the {\em coset leader} of $C_j$. For a  cyclic code $\mathcal{C}$,  its generator polynomial $g(x)$ can be written as  $g(x)=\prod_j\prod_{e\in C_j}(x-\alpha^e)$, where $j$ runs through some coset leaders of the $p$-cyclotomic cosets $C_j$ modulo $n.$  The set   $T=\bigcup_jC_j$ is referred to as the {\em defining set} of $\mathcal{C}$, which is the union of these $p$-cyclotomic cosets.


Below, let $\alpha$ be a primitive element of $\mathbb{F}_q$, $h_i(x)$ be the minimal polynomials of $\alpha^i$ over $\mathbb{F}_p$ and $Tr$ denote the trace function from $\mathbb{F}_q$ onto $\mathbb{F}_p$ given by
$$Tr(x)=x+x^p+\cdots+ x^{p^{m-1}}, \mathrm{~for~ any~} x\in \mathbb{F}_q.$$
Denote  $\mathbb{C}$ by the cyclic code with length $n=q-1$ and parity-check polynomial $h_1(x)h_{p^l+1}(x)h_{p^{3l}+1}(x)$ related to a class of Dembowski-Ostrom functions \cite{LLX09}, where $1\leq l\leq m-1$ and $l\notin \{\frac{m}{6}, \frac{m}{4}, \frac{m}{2}, \frac{3m}{4}, \frac{5m}{6}\}.$
  We introduce the linear code ${\overline{{\mathbb{C}}^{\bot}}}^{\bot}$ by
\begin{eqnarray}\label{code-1}
{\overline{{\mathbb{C}}^\bot}}^\bot:=\{ (Tr(ax^{p^{3l}+1}+bx^{p^l+1} &+& cx)_{x\in \mathbb{F}_q}+h ):\\
& & a, b , c \in \mathbb{F}_q, h\in\mathbb{F}_p \}. \nonumber
\end{eqnarray}
In \cite{LLX09},  Luo \emph{et al.} examined the weight distribution of cyclic code $\mathbb{C}$. In this paper, we will determine the weight distribution of the  linear code ${\overline{{\mathbb{C}}^\bot}}^\bot$ and then  obtain infinitely families $2$-designs from the  codewords with nonzero weight in ${\overline{{\mathbb{C}}^\bot}}^\bot$. If $m=3d$, ${\overline{{\mathbb{C}}^\bot}}^\bot$ is the code we discussed in \cite{DWTW190}, thus we suppose $m\neq3d$ in this paper.

\section{Preliminaries}\label{section-2}
In this section, we summarize some standard notation and basic facts on affine-invariant codes, $2$-designs, exponential sums and cyclotomic fields. 
\subsection{Some notation}
For convenience, we adopt the following notation unless otherwise stated in this paper.
\begin{itemize}
\item $m\geq3$ and $l$ are positive integers with $1\leq l\leq m-1$ and  $l\notin \{\frac{m}{6}, \frac{m}{4}, \frac{m}{2}, \frac{3m}{4}, \\\frac{5m}{6}\}$.
 \item  $d'=gcd(m, 2l)$ and  $d=gcd(m, l)$.
  \item  If $m$ is even, $s=m/2$ and  $\mu=(-1)^{s/d}$.
  \item $p$ is an odd prime and  $p^*=(-1)^{\frac{p-1}{2}}p$.
\item  $q=p^m$, $n=p^m-1$ and $\varepsilon=(-1)^\upsilon$, where $\upsilon=0,1$.
\item $\zeta_N=e^{2\pi \sqrt{-1}/N}$ is a primitive $N$-th root of unity for any  integer $N \ge 2$.
\item $\eta$ and $\eta'$ are the quadratic characters of $\mathbb{F}_q^*$ and $\mathbb{F}_p^*,$ respectively (we will give the definition below).
\end{itemize}
\subsection{Affine-invariant codes and $2$-designs}
We begin this subsection by introducing the affine-invariant code.
The set of coordinate permutations that map a code $\mathcal{C}$ to itself forms a group, which we call { \em the permutation automorphism group of $\mathcal{C}$} and denote by $PAut(\mathcal{C})$. We define the affine group $GA_1(\mathbb{F}_q)$ by the set of all permutations
$$\sigma_{a,b}: x\mapsto  ax+b $$
of $\mathbb{F}_q$, where $a \in \mathbb{F}_q^*$ and  $b \in \mathbb{F}_q.$
An {\em affine-invariant} code is an extended cyclic code $\overline {\mathcal{C}}$ over $\mathbb{F}_p$ such that $GA_1(\mathbb{F}_q)\subseteq PAut(\overline{\mathcal{C}})$~\cite{HP03}.

The $p$-adic expansion of each $s\in\mathcal{P}$ is given by
$$
s=\sum^{m-1}_{i=0}s_ip^i,~ ~0\leq s_i\leq p-1 ,~0\leq i \leq m-1.
$$
For any $r=\sum^{m-1}_{i=0}r_ip^i \in\mathcal{P}$, we say that $r\preceq s$ if $r_i \leq s_i$ for all $0\leq i\leq m-1$. Clearly,
$r\preceq s$ implies that $r \le s $.


For some applications, an important question is to decide whether a given linear codes $\overline{\mathcal{C}}$ and ${\overline{\mathcal{C}}}^\bot$ are affine-invariant or not. The answers are provided by the following lemma. More precisely, the first item given by Kasami et al. \cite{KLP67} points out that one can determine whether a given extended primitive cyclic code $\overline{\mathcal{C}}$ is affine-invariant by analysing the properties of its defining set and the second item proposed by Ding \cite{Ding18b} shows that one can get new affine-invariant code from the known ones.


\begin{lemma}\label{Kasami-Lin-Peterson}
Let $\overline{\mathcal{C}}$ be an extended cyclic code of length $p^m$ over $\mathbb{F}_p$ with defining set $\overline{T}$.
\begin{itemize}
\item[(1)] \cite{KLP67} $\overline{\mathcal{C}}$ is affine-invariant if and only if whenever $s \in   \overline{T}$  then $r \in \overline{T}$ for all $r \in \mathcal{P}$  with $r \preceq s$.

\item[(2)] \cite{Ding18b}  The dual of an affine-invariant code $\overline{\mathcal{C}}$ over $\mathbb{F}_p$ of length $n+1$ is also affine-invariant.
 \end{itemize}
\end{lemma}
%
%


\begin{theorem}\label{2-design}
  \cite{Ding18b} For each $i$ with $A_i \neq 0$ in an affine-invariant code $\overline{\mathcal{C}}$, the supports of the codewords of weight $i$ form a $2$-design.
\end{theorem}

Theorem \ref{2-design} is very attractive in the sense that they determine the existence of $2$-designs.  The following theorem, describe the relation of all codewords with the same support in a linear code $\mathcal{C}$,  will be used together with Eq.(\ref{condition}) to calculate the parameters of $2$-designs later.

\begin{theorem}\label{design parameter}~\cite{Ding18b} Let $\mathcal{C}$ be a linear code over $\mathbb{F}_p$ with length $n$ and minimum weight $\delta$. Let $w$ be the largest integer with $w\leq n$ satisfying
 $${w-\lfloor\frac{w+p-2}{p-1}\rfloor}< \delta.$$
Let $\mathbf{c_1}$ and $\mathbf{c_2}$ be two codewords of weight $i$ with $ \delta \leq i\leq w$ and $Suppt(\mathbf{c_1})=Suppt(\mathbf{c_2}).$ Then $\mathbf{c_1}=a\mathbf{c_2}$ for some $a\in \mathbb{F}_p^*$.
\end{theorem}

\subsection{Exponential sums} Generally,  the weight of each codeword in a cyclic code $\mathcal{C}$  can be expressed by certain exponential sums so that the weight distribution of $\mathcal{C}$ can be determined when the exponential sums could be computed explicitly (see~\cite{FY05,TXF16,Vlugt95,X12} and the references therein).

An additive character of $\mathbb{F}_q$ is a nonzero function $\chi$ from $\mathbb{F}_q$ to the set of complex numbers of absolute value $1$ such that $\chi(x+y)=\chi(x)\chi(y)$ for any pair $(x,y) \in \mathbb{F}_q^2$. For each $u \in \mathbb{F}_q$, the function
$$\chi_u(v)=\zeta_p^{Tr(uv)},~v \in \mathbb{F}_q,$$
denotes an additive character of $\mathbb{F}_q$. 
We call $\chi_1$ the {\em canonical} additive character of $\mathbb{F}_q$ and we have  $\chi_u(x)=\chi_1(ux)$ for all $u\in\mathbb{F}_q$ ~\cite{LN97}.

For each $j=0,1,\ldots,q-2,$ a multiplicative character  of $\mathbb{F}_q$  is defined by the function $\psi_j(\alpha^k)=\zeta_{q-1}^{jk},$ for $k=0,1,\ldots,q-2.$  We denote by $\eta:=\psi_{\frac{q-1}{2}}$ which  is called the quadratic character of $\mathbb{F}_q.$ Similarly,  we may define the quadratic character  $\eta ^{'}$  of $\mathbb{F}_p.$  We extend these quadratic characters by
    letting $\eta(0)=0$ and $\eta'(0)=0$.

The {\em Gauss sum} $G(\eta', \chi'_1)$ over $\mathbb{F}_p$ is defined by $G(\eta', \chi'_1)=\sum\limits_{v \in \mathbb{F}_p^*}\eta'(v)\chi'_1(v)=\sum\limits_{v \in \mathbb{F}_p}\eta'(v)\chi'_1(v)$, where $\chi'_1$ is the canonical additive character of $\mathbb{F}_p$. The following Lemma \ref{Gauss} related to Gauss sum is essential to determine the value of Eq.(\ref{S(a.b.c)}).
\begin{lemma}\label{Gauss}~\cite{LN97} With the notation above, we have
$$G(\eta', \chi'_1)={\sqrt{(-1)}^{(\frac{p-1}{2})^2}}\sqrt{p}=\sqrt{p^*}.$$
\end{lemma}

To determine the weight distribution of the code ${\overline{{\mathbb{C}}^{\bot}}}^{\bot}$, 
we introduce the following function
\begin{equation}\label{S(a.b.c)}
S(a,b,c)=\sum\limits_{x \in \mathbb{F}_q}\zeta_p^{Tr(ax^{p^{3l}+1}+bx^{p^l+1}+cx)},\qquad a, b, c \in \mathbb{F}_q.
\end{equation}

For the value distribution of $S(a,b,c)$, we have the following results given by Luo {\em et al.}.
\begin{lemma}\label{value distribution}~\cite{LLX09}
For $m \geq 3$ and $j\in \mathbb{F}_p^*$, the value distribution of $\{S(a,b,c) : a, b, c \in \mathbb{F}_q\}$ defined in Eq.(\ref{S(a.b.c)}) is given in Table \ref{1} if $d'=d$ is odd,  in  Table \ref{2} if $d'=d$ is even and  in  Table \ref{3} if $d'=2d$, respectively (see Tables~\ref{1}, ~\ref{2} and~\ref{3} in Appendix I).
\end{lemma}
It is clear that $d'=d$ is odd implies $m$ is odd.

\subsection{Cyclotomic fields}
We state the following basic facts on Galois group of cyclotomic fields $\mathbb{Q}(\zeta_p)$ since  $S(a,b,c)$ is element in $\mathbb{Q}(\zeta_p)$.

\begin{lemma}\label{Cyclotomic fields}~\cite{IR90}
  Let $\mathbb{Z}$ be the rational integer ring and $\mathbb{Q}$ be the rational field.

(1) The ring of integers in $K=\mathbb{Q}(\zeta_p)$ is $\mathcal{O}_k=\mathbb{Z}[\zeta_p]$ and $\{\zeta^i_{p} : 1 \leq i \leq p-1\}$ is an integral basis of $\mathcal{O}_k.$

(2) The filed extension $K/\mathbb{Q}$ is Galois of degree $p-1$ and the Galois group $Gal(K/\mathbb{Q})=\{\sigma_y : y \in (\mathbb{Z}/p\mathbb{Z})^*\},$ where the automorphism $\sigma_y$ of $K$ is defined by $\sigma_y(\zeta_p)=\zeta^y_{p}.$

(3) $K$ has a unique quadratic subfield $L=\mathbb{Q}(\sqrt{p^*}).$ For $1 \leq y \leq p-1, \sigma_y(\sqrt{p^*})
=\eta'(y)\sqrt{p^*}.$ Therefore, the Galois group $Gal(L/\mathbb{Q})$ is $\{1,\sigma_\gamma\},$ where $\gamma$ is any quadratic nonresidue in $\mathbb{F}_p.$
\end{lemma}

\section{
Main results}\label{section-3}
In the section, we only present the main results on the weight distribution of the code ${\overline{{\mathbb{C}}^{\bot}}}^{\bot}$  and the corresponding $2$-designs.
The proofs will be presented in subsequent section.


\begin{theorem}\label{weight1}
For $m\geq 6$, the weight distribution of the code ${\overline{{\mathbb{C}}^{\bot}}}^{\bot}$ over $\mathbb{F}_p$ with length $n+1$ and dimension $dim({\overline{{\mathbb{C}}^{\bot}}}^{\bot})=3m+1$ is given in Table \ref{4} when $d'=d$ is odd, in  Table \ref{5} when $d'=d$ is even and in  Table \ref{6} when $d'=2d$, respectively.
\begin{table}
\begin{center}
\caption{The weight distribution of ${\overline{{\mathbb{C}}^{\bot}}}^{\bot}$ when $d'=d$ is odd}\label{4}
\begin{tabular}{ll}
\hline\noalign{\smallskip}
Weight  &  Multiplicity   \\
\noalign{\smallskip}
\hline\noalign{\smallskip}
$0$  &  $1$ \\
$p^{m-1}(p-1)$  &  $ p(p^{2m+2d-1}+p^{2m+d}-p^{2m+d-1}$\\ &$-p^{2m}-p^{2m-1}+p^{2m-2d}-p^{2m-3d}$\\ &$+p^{2m-3d-1}+p^{m+2d-1}-p^m+$\\
&$p^{m-2d}-p^{m-2d-1}+p^{2d}-1)$\\ &$(p^m-1)/(p^{2d}-1) $    \\
$p^{\frac{m-1}{2}}(p^{\frac{m+1}{2}}-p^{\frac{m-1}{2}}+\varepsilon(-1)^\frac{p-1}{2}) $ & $ \frac{p^{m+2d}(p-1)(p^m-p^{m-d}-p^{m-2d}+1)(p^m-1)}{2(p^{2d}-1)}  $     \\
$ p^{\frac{m+d-2}{2}}(p-1)(p^{\frac{m-d}{2}}+\varepsilon)$  &  $ \frac{1}{2}p^{\frac{3m-3d}{2}}(p^{\frac{m-d}{2}}-\varepsilon)(p^m-1)$     \\
$ p^{\frac{m+d-2}{2}}(p^{\frac{m-d+2}{2}}-p^{\frac{m-d}{2}}+\varepsilon)$  &  $\frac{1}{2}p^{\frac{3m-3d}{2}}(p-1)(p^{\frac{m-d}{2}}+\varepsilon)(p^m-1)$     \\
$p^{\frac{m+2d-1}{2}}(p^{\frac{m-2d+1}{2}}-p^{\frac{m-2d-1}{2}}+\varepsilon(-1)^\frac{p-1}{2})$  &  $ \frac{p^{m-2d}(p-1)(p^{m-d}-1)(p^m-1)}{2(p^{2d}-1)}$     \\
$ p^m$  &  $ p-1$     \\
\noalign{\smallskip}
\hline
\end{tabular}
\end{center}
\end{table}

\begin{table}
\begin{center}
\caption{The weight distribution of ${\overline{{\mathbb{C}}^{\bot}}}^{\bot}$ when $d'=d$ is even}\label{5}
\begin{tabular}{ll}
\hline\noalign{\smallskip}
Weight  &  Multiplicity   \\
\noalign{\smallskip}
\hline\noalign{\smallskip}
$0$  &  $1$ \\
$p^{s-1}(p-1)(p^s+\varepsilon)$  &  $ \frac{p^{m+2d}(p^m-p^{m-d}-p^{m-2d}+1)(p^m-1)}{2(p^{2d}-1)}$     \\
$  p^{s-1}(p^{s+1}-p^s+\varepsilon)$  &  $ \frac{p^{m+2d}(p-1)(p^m-p^{m-d}-p^{m-2d}+1)(p^m-1)}{2(p^{2d}-1)}$    \\
$p^{\frac{m+d-2}{2}}(p-1)(p^{\frac{m-d}{2}}+\varepsilon)$  &  $\frac{1}{2}p^{\frac{3m-3d}{2}}(p^{\frac{m-d}{2}}-\varepsilon)(p^m-1)$     \\
$p^{\frac{m+d-2}{2}}(p^{\frac{m-d+2}{2}}-p^{\frac{m-d}{2}}+\varepsilon)$  &  $\frac{1}{2}p^{\frac{3m-3d}{2}}(p-1)(p^{\frac{m-d}{2}}+\varepsilon)(p^m-1)$     \\
$ p^{\frac{m+2d-2}{2}}(p-1)(p^{\frac{m-2d}{2}}+\varepsilon)$  &  $\frac{1}{2}p^{m-2d}(p^{m-d}-1)(p^m-1)/(p^{2d}-1)$     \\
$p^{\frac{m+2d-2}{2}}(p^{\frac{m-2d+2}{2}}-p^{\frac{m-2d}{2}}+\varepsilon)$  &  $\frac{p^{m-2d}(p-1)(p^{m-d}-1)(p^m-1)}{2(p^{2d}-1)}$     \\
$p^{m-1}(p-1)$  &  $ p(p^{2m-d}-p^{2m-2d}+p^{2m-3d}-p^{m-2d}$\\ &$+1)(p^m-1) $    \\
$ p^m$  &  $ p-1$     \\
\noalign{\smallskip}
\hline
\end{tabular}
\end{center}
\end{table}

\begin{table}
\begin{center}
\caption{The weight distribution of ${\overline{{\mathbb{C}}^{\bot}}}^{\bot}$ when $d'=2d$}\label{6}
\begin{tabular}{ll}
\hline\noalign{\smallskip}
Weight  &  Multiplicity   \\
\noalign{\smallskip}
\hline\noalign{\smallskip}
$0$  &  $1$ \\
$p^{s-1}(p-1)(p^s-\mu) $ & $p^m(p^{m+6d}-p^{m+4d}-p^{m+d}+\mu p^{s+5d}-$\\&$\mu p^{s+4d}+p^{6d})(p^m-1)/(p^d+1)(p^{2d}-1)$\\&$(p^{3d}+1)  $     \\
$ p^{s-1}(p^{s+1}-p^s+\mu)$  &  $p^m(p-1)(p^{m+6d}-p^{m+4d}-p^{m+d}+\mu p^{s+5d}-$\\&$\mu p^{s+4d}+p^{6d})(p^m-1)/(p^d+1)(p^{2d}-1)$\\&$(p^{3d}+1)  $     \\
$p^{s+d-1}(p-1)(p^{s-d}+\mu)$  &  $ p^{m-2d}(p^{m+3d}+p^{m+2d}-p^m-p^{m-d}-p^{m-2d}-$\\&$\mu p^{s+3d}+\mu p^s+p^{3d})(p^m-1)/(p^d+1)^2(p^{2d}-1)$     \\
$  p^{s+d-1}(p^{s-d+1}-p^{s-d}-\mu)$  &  $ p^{m-2d}(p-1)(p^{m+3d}+p^{m+2d}-p^m-p^{m-d}-$\\&$p^{m-2d}-\mu p^{s+3d}+\mu p^s+p^{3d})(p^m-1)/(p^d+1)^2$\\&$(p^{2d}-1)$    \\
$p^{s+2d-1}(p-1)(p^{s-2d}-\mu)$  &  $p^{m-4d}(p^{s-d}+\mu)(p^{s+d}+p^s-p^{s-2d}-\mu p^d)$\\&$(p^m-1)/(p^d+1)^2(p^{2d}-1)$     \\
$p^{s+2d-1}(p^{s-2d+1}-p^{s-2d}+\mu)$  &  $p^{m-4d}(p-1)(p^{s-d}+\mu)(p^{s+d}+p^s-p^{s-2d}-$\\&$\mu p^d)(p^m-1)/(p^d+1)^2(p^{2d}-1)$     \\
$p^{s+3d-1}(p-1)(p^{s-3d}+\mu)$  &  $ p^{m-6d}(p^{s-2d}-\mu)(p^{s-d}+\mu)(p^m-1)/(p^d+1)$\\&$(p^{2d}-1)(p^{3d}+1)$     \\
$p^{s+3d-1}(p^{s-3d+1}-p^{s-3d}-\mu)$  &  $ p^{m-6d}(p-1)(p^{s-2d}-\mu)(p^{s-d}+\mu)(p^m-1)/$\\&$(p^d+1)(p^{2d}-1)(p^{3d}+1)$    \\
$p^{m-1}(p-1)$  &  $ p(1-\mu p^{3s-d}-\mu p^{3s-8d}+p^{m-d}+$\\&$\frac{p^{2m}+p^{2m-9d}+\mu p^{3s-3d}-\mu p^{3s-5d}-p^{m-4d}-p^{m-6d}}{p^d+1})(p^m-1) $    \\
$ p^m$  &  $ p-1$     \\
\noalign{\smallskip}
\hline
\end{tabular}
\end{center}
\end{table}
\end{theorem}

  One can see that the code is at most $10$-weight if $d'=d$ is odd or $d'=2d$, at most $14$-weight if $d'=d$ is even. 



\begin{theorem}\label{$2-$design-1}
Let $m\geq 6$ be a positive integer. Then the supports of the all codewords of weight $i>0$ in ${\overline{{\mathbb{C}}^{\bot}}}^{\bot}$ form a $2$-design, provided that $A_i\neq0.$
\end{theorem}

For the parameters of the 2-designs derived from ${\overline{{\mathbb{C}}^{\bot}}}^{\bot}$, we have the following.

\begin{theorem}\label{parameter-1}
Let $\mathcal{B}$ be the set of the supports of the codewords of ${\overline{{\mathbb{C}}^{\bot}}}^{\bot}$ with weight $i,$ where $A_i\neq 0.$ Then for $m\geq 6$, ${\overline{{\mathbb{C}}^{\bot}}}^{\bot}$ holds $2$-$(p^m, i, \lambda)$ designs for the following pairs:

(1) If $d'=d$ is odd, then
\begin{itemize}
\item $(i, \lambda)=(p^{m-1}(p-1), (p^m-p^{m-1}-1)(p^{2m+2d-1}+p^{2m+d}-p^{2m+d-1}\\-p^{2m}-p^{2m-1}+p^{2m-2d}-p^{2m-3d}+p^{2m-3d-1}+
    p^{m+2d-1}-p^m+p^{m-2d}-p^{m-2d-1}+p^{2d}-1)/(p^{2d}-1));$
\item $(i, \lambda)=(p^{\frac{m-1}{2}}(p^{\frac{m+1}{2}}-p^{\frac{m-1}{2}}+\varepsilon(-1)^{\frac{p-1}{2}}), p^{\frac{m+4d-1}{2}}(p^{\frac{m+1}{2}}-p^{\frac{m-1}{2}}+\varepsilon\\(-1)^{\frac{p-1}{2}})[p^{\frac{m-1}{2}}
    (p^{\frac{m+1}{2}}-p^{\frac{m-1}{2}}+\varepsilon(-1)^{\frac{p-1}{2}})-1](p^m-p^{m-d}-p^{m-2d}+1)/2(p^{2d}-1));$
\item $(i, \lambda)=( p^{\frac{m+d-2}{2}}(p-1)(p^{\frac{m-d}{2}}+\varepsilon), \frac{1}{2}p^{m-d-1}[p^{\frac{m+d-2}{2}}(p-1)(p^{\frac{m-d}{2}}+\varepsilon)-1](p^{m-d}-1));$
\item $(i, \lambda)=( p^{\frac{m+d-2}{2}}(p^{\frac{m-d+2}{2}}-p^{\frac{m-d}{2}}+\varepsilon), \frac{1}{2}p^{m-d-1}(p^{\frac{m-d+2}{2}}-p^{\frac{m-d}{2}}+\varepsilon)[\\p^{\frac{m+d-2}{2}}(p^{\frac{m-d+2}{2}}
    -p^{\frac{m-d}{2}}+\varepsilon)-1](p^{\frac{m-d}{2}}+\varepsilon));$
\item$(i,\lambda)=(p^{\frac{m+2d-1}{2}}(p^{\frac{m-2d+1}{2}}-p^{\frac{m-2d-1}{2}}+\varepsilon(-1)^{\frac{p-1}{2}}), \frac{1}{2}p^{\frac{m-2d-1}{2}}(p^{\frac{m-2d+1}{2}}-p^{\frac{m-2d-1}{2}}+\varepsilon(-1)^{\frac{p-1}{2}})
    [p^{\frac{m+2d-1}{2}}(p^{\frac{m-2d+1}{2}}-p^{\frac{m-2d-1}{2}}+\varepsilon(-1)^{\frac{p-1}{2}})-1]\\(p^{m-d}-1)/
    (p^{2d}-1));$
\end{itemize}
(2) If $d'=d$ is even and $m\geq 8$, then
\begin{itemize}
\item $(i, \lambda)=(p^{s-1}(p-1)(p^s+\varepsilon), p^{s+2d-1}(p^s+\varepsilon)[p^{s-1}(p-1)(p^s+1)-1](p^m-p^{m-d}-p^{m-2d}+1)/2(p^{2d}-1));$
\item $(i, \lambda)=(p^{s-1}(p^{s+1}-p^s+\varepsilon), \frac{1}{2}p^{s+2d-1}(p^{s+1}-p^s+\varepsilon) [p^{s-1}(p^{s+1}-p^s+\varepsilon)-1](p^m-p^{m-d}-p^{m-2d}+1)/(p^{2d}-1));$
\item $(i, \lambda)=(p^{\frac{m+d-2}{2}}(p-1)(p^{\frac{m-d}{2}}+\varepsilon), \frac{1}{2}p^{m-d-1}(p^{m-d}-1)[p^{\frac{m+d-2}{2}}(p-1)(p^{\frac{m-d}{2}}+\varepsilon)-1]);$
\item $(i, \lambda)=(p^{\frac{m+d-2}{2}}(p^{\frac{m-d+2}{2}}-p^{\frac{m-d}{2}}+\varepsilon),  \frac{1}{2}p^{m-d-1}(p^{\frac{m-d+2}{2}}-p^{\frac{m-d}{2}}+\varepsilon)[\\p^{\frac{m+d-2}{2}}(p^{\frac{m-d+2}{2}}
    -p^{\frac{m-d}{2}}+\varepsilon)-1](p^{\frac{m-d}{2}}+\varepsilon));$
\item $(i, \lambda)=( p^{\frac{m+2d-2}{2}}(p-1)(p^{\frac{m-2d}{2}}+\varepsilon), \frac{1}{2}p^{s-d-1}(p^{\frac{m-2d}{2}}+\varepsilon)[p^{\frac{m+2d-2}{2}}(p-1)(p^{\frac{m-2d}{2}}+\varepsilon)-1](p^{m-d}
    -1)/(p^{2d}-1));$
\item $(i, \lambda)=(p^{\frac{m+2d-2}{2}}(p^{\frac{m-2d+2}{2}}-p^{\frac{m-2d}{2}}+\varepsilon), \frac{1}{2}p^{s-d-1}(p^{\frac{m-2d+2}{2}}-p^{\frac{m-2d}{2}}+\varepsilon)[p^{\frac{m+2d-2}{2}}(p^{\frac{m-2d+2}
    {2}}-p^{\frac{m-2d}{2}}+\varepsilon)-1](p^{m-d}-1)/(p^{2d}-1));$
\item $(i, \lambda)=(p^{m-1}(p-1), (p^m-p^{m-1}-1)(p^{2m-d}-p^{2m-2d}+p^{2m-3d}-p^{m-2d}+1)).$
\end{itemize}
For $m=6$, ${\overline{{\mathbb{C}}^{\bot}}}^{\bot}$ also holds $2$-$(p^m, i, \lambda)$ designs for the following pairs:
\begin{itemize}
\item $(i, \lambda)=(p^2(p-1)(p^3+\varepsilon), p^6(p^3+\varepsilon)[p^2(p-1)(p^3+\varepsilon)-1](p^6\\-p^4-p^2+1)/2(p^4-1));$
\item $(i, \lambda)=(p^2(p^4-p^3+\varepsilon), \frac{1}{2}p^6(p^4-p^3+\varepsilon) [p^2(p^4-p^3
    +\varepsilon)-1](p^6\\-p^4-p^2+1)/(p^4-1));$
\item $(i, \lambda)=(p^3(p-1)(p^2+\varepsilon), \frac{1}{2}p^3(p^4-1)[p^3(p-1)(p^2+\varepsilon)-1]);$
\item $(i, \lambda)=(p^3(p^3-p^2+\varepsilon),  \frac{1}{2}p^3(p^3-p^2+\varepsilon)[p^3(p^3
    -p^2+\varepsilon)-1](p^2\\+\varepsilon));$
\item $(i, \lambda)=(p^4(p-1)^2, \frac{1}{2}(p-1)[p^4(p-1)^2-1]);$
\item $(i, \lambda)=(p^4(p^2-p+\varepsilon), \frac{1}{2}(p^2-p+\varepsilon)[p^4(p^2-p+\varepsilon)-1]);$
\item $(i, \lambda)=(p^5(p-1), (p^6-p^5-1)(p^{10}-p^8+p^6-p^2+1)).$
\end{itemize}
(3) If $d'=2d$, then
\begin{itemize}
\item $(i, \lambda)=(p^{s-1}(p-1)(p^s-\mu), p^{s-1}(p^s-\mu)[p^{s-1}(p-1)(p^s-\mu)-1](p^{m+6d}-p^{m+4d}-p^{m+d}+\mu p^{s+5d}-\mu p^{s+4d}+p^{6d})/(p^d+1)(p^{2d}-1)(p^{3d}+1));$
\item $(i, \lambda)=(p^{s-1}(p^{s+1}-p^s+\mu), p^{s-1}(p^{s+1}-p^s+\mu)[p^{s-1}(p^{s+1}-p^s+\mu)-1](p^{m+6d}-p^{m+4d}-p^{m+d}+\mu p^{s+5d}-\mu p^{s+4d}+p^{6d})/(p^d+1)\\(p^{2d}-1)(p^{3d}+1));$
\item $(i, \lambda)=(p^{s+d-1}(p-1)(p^{s-d}+\mu), p^{s-d-1}(p^{s-d}+\mu)[p^{s+d-1}(p-1)(\\p^{s-d}+\mu)-1](p^{m+3d}+p^{m+2d}-p^m-p^{m-d}-p^{m-2d}-\mu p^{s+3d}+\mu p^s+p^{3d})/(p^d+1)^2(p^{2d}-1));$
\item $(i, \lambda)=(p^{s+d-1}(p^{s-d+1}-p^{s-d}-\mu),  p^{s-d-1}(p^{s-d+1}-p^{s-d}-\mu)[\\p^{s+d-1}(p^{s-d+1}-p^{s-d}-\mu)-1](p^{m+3d}+p^{m+2d}-p^m-p^{m-d}
    -\\p^{m-2d}-\mu p^{s+3d}+\mu p^s+p^{3d})/(p^d+1)^2(p^{2d}-1));$
\item $(i, \lambda)=(p^{s+2d-1}(p-1)(p^{s-2d}-\mu), p^{s-2d-1}(p^{s-2d}-\mu)[p^{s+2d-1}(\\(p-1)p^{s-2d}-\mu)-1](p^{s-d}+\mu)(p^{s+d}+p^s-p^{s-2d}-\mu p^d)/(p^d+1)^2(p^{2d}-1));$
\item $(i, \lambda)=(p^{s+2d-1}(p^{s-2d+1}-p^{s-2d}+\mu), p^{s-2d-1}(p^{s-2d+1}-p^{s-2d}+\mu)[p^{s+2d-1}(p^{s-2d+1}-p^{s-2d}+\mu)-1](p^{s-d}+\mu)(p^{s+d}+p^s-p^{s-2d}-\mu p^d)/(p^d+1)^2(p^{2d}-1));$
\item $(i, \lambda)=(p^{s+3d-1}(p-1)(p^{s-3d}+\mu), p^{s-3d-1}(p^{s-3d}+\mu)[p^{s+3d-1}(\\(p-1)p^{s-3d}+\mu)-1](p^{s-2d}-\mu)(p^{s-d}+\mu)/(p^d+1)(p^{2d}-1)(p^{3d}+1));$
\item $(i, \lambda)=(p^{s+3d-1}(p^{s-3d+1}-p^{s-3d}-\mu), p^{s-3d-1}(p^{s-3d+1}-p^{s-3d}-\mu)[p^{s+3d-1}(p^{s-3d+1}-p^{s-3d}-\mu)-1](p^{s-2d}-\mu)(p^{s-d}+\mu)/(p^d+1)(p^{2d}-1)(p^{3d}+1));$
\item $(i, \lambda)=(p^{m-1}(p-1), (p^m-p^{m-1}-1)(1-\mu p^{3s-d}-\mu p^{3s-8d}+p^{m-d}\\+\frac{p^{2m}+p^{2m-9d}+\mu p^{3s-3d}-\mu p^{3s-5d}-p^{m-4d}-p^{m-6d}}{p^d+1})).$
\end{itemize}
\end{theorem}


\section{Proofs of the main results}\label{section-4}
Now, we are ready to present the proofs of our main results. We begin this section by proving the weight distribution of the code ${\overline{{\mathbb{C}}^{\bot}}}^{\bot}$.
\begin{proof}[Proof of Theorem \ref{weight1}]
For each nonzero codeword $\mathbf{c}(a,b,c,h)=(c_0, c_1,\ldots, c_n)$ in ${\overline{{\mathbb{C}}^{\bot}}}^{\bot},$ the Hamming weight of $\mathbf{c}(a,b,c,h)$ is
\begin{equation}\label{weight formula-1}
w_H(\mathbf{c}(a,b,c,h))=p^m-T(a,b,c,h),
\end{equation}
where
 \begin{eqnarray*}
 T(a,b,c,h)=|\{x\in \mathbb{F}_q:&  Tr(ax^{p^{3l}+1}+bx^{p^l+1}+cx)+h=0,\\
              &a,b,c \in \mathbb{F}_q, h\in \mathbb{F}_p\}|.
 \end{eqnarray*}
 With the  orthogonality relations for character sums, and the definitions of $S(a,b,c)$ and $\sigma_y$, we have the following equation
\begin{eqnarray}\label{T(a,b,c,h)}
T(a,b,c,h)&=&\frac{1}{p}\sum\limits_{y \in \mathbb{F}_p}\sum\limits_{x \in \mathbb{F}_q}\zeta_p^{y[Tr(ax^{p^{3l}+1}+bx^{p^l+1}+cx)+h]}\nonumber\\
&=&\frac{1}{p}\sum\limits_{y \in \mathbb{F}_p}\zeta_p^{yh}\sum\limits_{x \in \mathbb{F}_q}\zeta_p^{yTr(ax^{p^{3l}+1}+bx^{p^l+1}+cx)}\nonumber\\
&=&p^{m-1}+\frac{1}{p}\sum\limits_{y \in \mathbb{F}_p^*}\zeta_p^{yh}\sigma_y(S(a,b,c)).
\end{eqnarray}

If $d'=d$ is odd, with the value of $S(a,b,c)$ given in  Lemma \ref{value distribution} and the result on cyclotomic fields presented in Lemma \ref{Cyclotomic fields}, we obtain explicitly $\sigma_y(S(a,b,c))$. Therefore,  plugging $\sigma_y(S(a,b,c))$ in the equation above and applying Lemma \ref{Gauss}, we have the corresponding value distribution of $T(a,b,c,h)$ given in Table \ref{8} after a tedious calculation, where $j\in\mathbb{F}_p$ which is different from that of Lemma \ref{value distribution}.  Thus, with Table \ref{8} and Eq.(\ref{weight formula-1}), we can present the weight distribution of the code ${\overline{{\mathbb{C}}^{\bot}}}^{\bot}$ in Table \ref{9}.

\begin{table}
\begin{center}
\caption{The value of $T(a,b,c,h)$ when $d'=d$ is odd ($j\in \mathbb{F}_p$)}\label{8}
\begin{tabular}{ll}
\hline\noalign{\smallskip}
Value  &  Corresponding Condition   \\
\noalign{\smallskip}
\hline\noalign{\smallskip}
$p^{m-1}$ &  $S(a,b,c)=0~ \mathrm{or}~ S(a,b,c)=\varepsilon\sqrt{p^*}\zeta_p^jp^{\frac{m-1}{2}}~ \mathrm{and}~ h+j $ \\&$= 0~\mathrm{or}~ S(a,b,c)=\varepsilon\sqrt{p^*}\zeta_p^jp^{\frac{m+2d-1}{2}}~ \mathrm{and}~ h+j = 0 $ \\
$p^{m-1}+\varepsilon(-1)^{\frac{p-1}{2}}p^{\frac{m-1}{2}} $ &  $ S(a,b,c)=\pm\varepsilon\sqrt{p^*}p^{\frac{m-1}{2}}\zeta_p^j~\mathrm{and}~ \eta'(h+j)\in\{\varepsilon, -\varepsilon\} $\\
$p^{m-1}+\varepsilon p^{\frac{m+d-2}{2}}(p-1) $  &  $ S(a,b,c)=\varepsilon p^{\frac{m+d}{2}}\zeta_p^j~\mathrm{and}~ h+j=0$    \\
$ p^{m-1}-\varepsilon p^{\frac{m+d-2}{2}} $  &  $S(a,b,c)=\varepsilon p^{\frac{m+d}{2}}\zeta_p^j~\mathrm{and}~ h+j\neq0$     \\
$p^{m-1}+\varepsilon(-1)^{\frac{p-1}{2}}p^{\frac{m+2d-1}{2}}$  &  $S(a,b,c)=\pm\varepsilon \sqrt{p^*}p^{\frac{m+2d-1}{2}}\zeta_p^j~\mathrm{and}~ \eta'(h+j)\in\{\varepsilon, -\varepsilon\}$\\
$p^m$ & $S(a,b,c)=p^m~ \mathrm{and}~ h = 0 $\\
$ 0 $  &  $ S(a,b,c)= p^m~\mathrm{and}~ h\neq0$      \\
\noalign{\smallskip}
\hline
\end{tabular}
\end{center}
\end{table}

\begin{table}
\begin{center}
\caption{The weight distribution of  ${\overline{{\mathbb{C}}^{\bot}}}^{\bot}$  when $d'=d$ is odd}\label{9}
\begin{tabular}{ll}
\hline\noalign{\smallskip}
Value  &  Multiplicity   \\
\noalign{\smallskip}
\hline\noalign{\smallskip}
$p^{m-1}(p-1)$  &  $M_{1\upsilon}+\frac{p-1}{2}M_{2\upsilon}+M_{5\upsilon}+\frac{p-1}{2}M_{6\upsilon}$\\&$+pM_7$ \\
$p^{\frac{m-1}{2}}(p^{\frac{m+1}{2}}-p^{\frac{m-1}{2}}+\varepsilon(-1)^{\frac{p-1}{2}})$  &  $\frac{p-1}{2}M_{1\upsilon}+\frac{(p-1)^2}{2}M_{2\upsilon}$    \\
$ p^{\frac{m+d-2}{2}}(p-1)(p^{\frac{m-d}{2}}-\varepsilon)$  &  $M_{3\upsilon}+(p-1)M_{4\upsilon}$     \\
$ p^{\frac{m+d-2}{2}}(p^{\frac{m-d+2}{2}}-p^{\frac{m-d}{2}}+\varepsilon)$  &  $(p-1)M_{3\upsilon}+(p-1)^2M_{4\upsilon}$  \\
$ p^{\frac{m+2d-1}{2}}(p^{\frac{m-2d+1}{2}}-p^{\frac{m-2d-1}{2}}+\varepsilon(-1)^{\frac{p-1}{2}})$ &  $ \frac{p-1}{2}M_{5\upsilon}+\frac{(p-1)^2}{2}M_{6\upsilon}$  \\
$p^m$ &  $ p-1 $  \\
\noalign{\smallskip}
\hline
\end{tabular}
\end{center}
\end{table}

With the similar proof idea, we can get  Table \ref{10} if   $d'=d$ is even, and
Table \ref{11} if  $d'=2d$, respectively.
\begin{table}
\begin{center}
\caption{The weight distribution of ${\overline{{\mathbb{C}}^{\bot}}}^{\bot}$ when $d'=d$ is even}\label{10}
\begin{tabular}{ll}
\hline\noalign{\smallskip}
Value  &  Multiplicity    \\
\noalign{\smallskip}
\hline\noalign{\smallskip}
$p^{s-1}(p-1)(p^s-\varepsilon)$  &  $M_{1\upsilon}+(p-1)M_{2\upsilon}$ \\
$p^{s-1}(p^{s+1}-p^s+\varepsilon)$  &  $ (p-1)M_{1\upsilon}+(p-1)^2M_{2\upsilon} $    \\
$p^{\frac{m+d-2}{2}}(p-1)(p^{\frac{m-d}{2}}-\varepsilon) $  &  $M_{3\upsilon}+(p-1)M_{4\upsilon}$     \\
$p^{\frac{m+d-2}{2}}(p^{\frac{m-d+2}{2}}-p^{\frac{m-d}{2}}+\varepsilon)$ &  $(p-1)M_{3\upsilon}+(p-1)^2M_{4\upsilon}$  \\
$p^{\frac{m+2d-2}{2}}(p-1)(p^{\frac{m-2d}{2}}-\varepsilon)$ &  $M_{5\upsilon}+(p-1)M_{6\upsilon}$  \\
$p^{\frac{m+2d-2}{2}}(p^{\frac{m-2d+2}{2}}-p^{\frac{m-2d}{2}}+\varepsilon)$ &  $(p-1)M_{5\upsilon}+(p-1)^2M_{6\upsilon} $  \\
$p^{m-1}(p-1)$ &  $pM_7$\\
$p^m$ &  $p-1 $  \\
\noalign{\smallskip}
\hline
\end{tabular}
\end{center}
\end{table}

\begin{table}
\begin{center}
\caption{The weight distribution of ${\overline{{\mathbb{C}}^{\bot}}}^{\bot}$ when $d'=2d$}\label{11}
\begin{tabular}{ll}
\hline\noalign{\smallskip}
Value  &  Multiplicity   \\
\noalign{\smallskip}
\hline\noalign{\smallskip}
$p^{s-1}(p-1)(p^s-\mu)$  &  $M_1+(p-1)M_2$ \\
$ p^{s-1}(p^{s+1}-p^s+\mu)$  &  $ (p-1)M_1+(p-1)^2M_2 $    \\
$p^{s+d-1}(p-1)(p^{s-d}+\mu)$  &  $M_3+(p-1)M_4 $     \\
$ p^{s+d-1}(p^{s-d+1}-p^{s-d}-\mu)$  &  $(p-1)M_3+(p-1)^2M_4$     \\
$p^{s+2d-1}(p-1)(p^{s-2d}-\mu) $  &  $M_5+(p-1)M_6$     \\
$p^{s+2d-1}(p^{s-2d+1}-p^{s-2d}+\mu)$  &  $(p-1)M_5+(p-1)^2M_6 $  \\
$p^{s+3d-1}(p-1)(p^{s-3d}+\mu)$ &  $M_7+(p-1)M_8$  \\
$p^{s+3d-1}(p^{s-3d+1}-p^{s-3d}-\mu)$ &  $(p-1)M_7+(p-1)^2M_8$  \\
$p^{m-1}(p-1)$ &  $pM_9$  \\
$p^m$ &  $p-1 $  \\
\noalign{\smallskip}
\hline
\end{tabular}
\end{center}
\end{table}

Thus, we complete the proof of Theorem \ref{weight1}.
\end{proof}


Now, we will use Lemma \ref{Kasami-Lin-Peterson} to  prove that $\overline{{\mathbb{C}}^{\bot}}$ is affine-invariant.

\begin{lemma}\label{affine invariant}
The extended code $\overline{{\mathbb{C}}^{\bot}}$ is affine-invariant.
\end{lemma}

\begin{proof}
As one know, $T =C_1 \cup C_{p^l+1}\cup C_{p^{3l}+1}$ is the defining set  of the cyclic code ${\mathbb{C}}^{\bot}$. Since $0 \not \in T$, the defining set $\overline{T}$ of $\overline{{\mathbb{C}}^{\bot}}$ is given by $\overline{T} = C_1 \cup C_{p^l+1}\cup C_{p^{3l}+1} \cup \{0\}$.
Let $v \in \overline{T} $ and $r \in \mathcal{P}$ such that  $r \preceq v$. To applying  Lemma \ref{Kasami-Lin-Peterson}, we  need to prove $r \in \overline{T}$.

If $r=0,$ it is easily seen that $r\in \overline{T}$. Consider now the case $r>0$. If $v \in  C_1$,  the Hamming weight $wt(v) = 1.$ It follows from the assumption $r \preceq v$ that  $wt(r) = 1,$ which  implies that  $r \in C_1 \subset  \overline{T}.$  If $v  \in C_{p^l+1} \cup C_{p^{3l}+1}$, then the Hamming weight $wt(v) = 2.$ Since $r \preceq v$, either $wt(r) = 1$ or $r =v.$ In both   cases, $r \in  \overline{T}.$ The desired conclusion then follows from Lemma \ref{Kasami-Lin-Peterson}.

\end{proof}




\begin{proof}[Proof of Theorem \ref{$2-$design-1}]
Theorem \ref{$2-$design-1} is an immediate result of Theorem \ref{2-design} and Lemma \ref{affine invariant}.
\end{proof}

\begin{proof}[Proof of Theorem \ref{parameter-1}]
By Theorem \ref{design parameter}, it is easy to know that the number of the different supports of all codewords with weight $i\neq 0$ in ${\overline{{\mathbb{C}}^{\bot}}}^{\bot}$ is equal to $A_i/(p-1)$ for each $i,$ where $A_i$ is given in Tables \ref{4}-\ref{6}. Then, from Theorem \ref{$2-$design-1} and  Eq.(\ref{condition}), together with a tedious calculation, we can get the desired conclusions.
\end{proof}

\begin{example}\label{example2}
If $(p, m, l)=(3, 6, 2)$, then the code ${\overline{{\mathbb{C}}^{\bot}}}^{\bot}$ has parameters $[729,19,\\324]$ and weight enumerator $1+3276z^{324}+6552z^{405}+2653560z^{432}+4245696\\z^{459}+171950688z^{468}+343901376z^{477}+116208456z
^{486}+343901376z^{495}+171950688z^{504}+5307120z^{513}+2122848z^{540}+6552z^{567}+3276z^{648}+2z^{729},$ which confirms  the results given in Theorem \ref{weight1}.
\end{example}

\section{Concluding remarks}\label{section-5}
In this paper, via using exponential sums, we first determined the weight distribution of a class of linear codes derived from the cyclic codes related to Dembowski-Ostrom functions. By the properties of affine-invariant codes, we then obtained that ${\overline{{\mathbb{C}}^{\bot}}}^{\bot}$ hold $2$-designs and determined their parameters explicitly. 
However, since the case  for $m=6$ and $i=p^4(p^2-1)$ does not meet the condition of Theorem  \ref{design parameter}, there remain some unsolved problems concerning the parameters of the $2$-designs derived from the supports of all codewords with the weight $i$ in ${\overline{{\mathbb{C}}^{\bot}}}^{\bot}$. This may constitute a challenge for future work.

\section*{Appendix I}
\begin{table}
\begin{center}
\caption{The value distribution of $S(a,b,c)$ when $d'=d$ is odd}\label{1}
\begin{tabular}{ll}
\hline\noalign{\smallskip}
Value  &  Multiplicity   \\
\noalign{\smallskip}
\hline\noalign{\smallskip}
$\varepsilon \sqrt{p^*}p^{\frac{m-1}{2}}$  &  $M_{1\upsilon}=\frac{ p^{m+2d-1}(p^m-p^{m-d}-p^{m-2d}+1)(p^m-1)}{2(p^{2d}-1)}$ \\
$\varepsilon\zeta^j_p\sqrt{p^*}p^{\frac{m-1}{2}}$  &  $ M_{2\upsilon}=\frac{p^{2d}(p^{m-1}+\varepsilon\eta'(-j)p^{\frac{m-1}{2}})(p^m-p^{m-d}-p^{m-2d}+1)
(p^m-1)}{2(p^{2d}-1)} $    \\
$\varepsilon p^{\frac{m+d}{2}}$  &  $M_{3\upsilon}=\frac{1}{2}p^{m-d-1}(p^{\frac{m-d}{2}}+\varepsilon(p-1))(p^{\frac{m-d}{2}}+\varepsilon)(p^m-1) $     \\
$ \varepsilon\zeta^j_p p^{\frac{m+d}{2}} $  &  $M_{4\upsilon}=\frac{1}{2}p^{m-d-1}(p^{\frac{m-d}{2}}-\varepsilon)(p^{\frac{m-d}{2}}+\varepsilon)(p^m-1)$     \\
$  \varepsilon \sqrt{p^*}p^{\frac{m+2d-1}{2}} $  &  $M_{5\upsilon}=\frac{1}{2}p^{m-2d-1}(p^{m-d}-1)/(p^{2d}-1)(p^m-1) $     \\
$ \varepsilon \zeta^j_p\sqrt{p^*}p^{\frac{m+2d-1}{2}}$  &  $M_{6\upsilon}=\frac{(p^{m-2d-1}+\varepsilon\eta'(-j)p^{\frac{m-2d-1}{2}})(p^{m-d}-1)(p^m-1)}{2(p^{2d}-1)} $  \\
$0$ &  $ M_7=(p^{2m-d}-p^{2m-2d}+p^{2m-3d}-p^{m-2d}+1)(p^m-1)$  \\
$p^m$ &  $ 1 $  \\
\noalign{\smallskip}
\hline
\end{tabular}
\end{center}
\end{table}

\begin{table}
\begin{center}
\caption{The value distribution of $S(a,b,c)$ when $d'=d$ is even}\label{2}
\begin{tabular}{ll}
\hline\noalign{\smallskip}
Value  &  Multiplicity   \\
\noalign{\smallskip}
\hline\noalign{\smallskip}
$\varepsilon p^s$  &  $M_{1\upsilon}= \frac{p^{2d}(p^{m-1}+\varepsilon(p-1)p^{s-1})(p^m-p^{m-d}-p^{m-2d}+1)(p^m-1)}{2(p^{2d}-1)}$ \\
$\varepsilon\zeta^j_p p^s$  &  $ M_{2\upsilon}=\frac{p^{2d}(p^{m-1}-\varepsilon p^{s-1})(p^m-p^{m-d}-p^{m-2d}+1)(p^m-1)}{2(p^{2d}-1)} $    \\
$ \varepsilon p^{\frac{m+d}{2}}$  &  $M_{3\upsilon}=\frac{1}{2}p^{m-d-1}(p^{\frac{m-d}
{2}}+\varepsilon(p-1))(p^{\frac{m-d}{2}}+\varepsilon)(p^m-1) $     \\
$ \varepsilon\zeta^j_pp^{\frac{m+d}{2}}$  &  $M_{4\upsilon}=\frac{1}{2}p^{m-d-1}(p^{\frac{m-d}{2}}-\varepsilon)(p^{\frac{m-d}{2}}+\varepsilon)(p^m-1)$     \\
$  \varepsilon p^{\frac{m+2d}{2}}  $  &  $M_{5\upsilon}=\frac{(p^{m-2d-1}+\varepsilon(p-1)p^{\frac{m-2d}{2}-1})(p^{m-d}-1)(p^m-1)}{2(p^{2d}-1)} $     \\
$ \varepsilon \zeta^j_p p^{\frac{m+2d}{2}}$  &  $M_{6\upsilon}=\frac{(p^{m-2d-1}-\varepsilon p^{\frac{m-2d}{2}-1})(p^{m-d}-1)(p^m-1)}{2(p^{2d}-1)} $  \\
$0$ &  $ M_7=(p^{2m-d}-p^{2m-2d}+p^{2m-3d}-p^{m-2d}+1)(p^m-1) $  \\
$p^m$ &  $ 1 $  \\
\noalign{\smallskip}
\hline
\end{tabular}
\end{center}
\end{table}

\begin{table}
\begin{center}
\caption{The value distribution of $S(a,b,c)$ when $d'=2d$}\label{3}
\begin{tabular}{ll}
\hline\noalign{\smallskip}
Value  &  Multiplicity   \\
\noalign{\smallskip}
\hline\noalign{\smallskip}
$\mu p^s$  &  $ M_1=(p^{m-1}+\mu (p-1)p^{s-1})(p^{m+6d}-p^{m+4d}-p^{m+d}+\mu p^{s+5d}$\\&$-\mu p^{s+4d}+p^{6d})(p^m-1)/(p^d+1)(p^{2d}-1)(p^{3d}+1) $    \\
$ \mu\zeta^j_pp^s $  &  $M_2=(p^{m-1}-\mu p^{s-1})(p^{m+6d}-p^{m+4d}-p^{m+d}+\mu p^{s+5d}$\\&$-\mu p^{s+4d}+p^{6d})(p^m-1)/(p^d+1)(p^{2d}-1)(p^{3d}+1) $     \\
$ -\mu p^{s+d}  $  &  $M_3=(p^{m-2d-1}-\mu (p-1)p^{s-d-1})(p^{m+3d}+p^{m+2d}-p^m-$\\&$p^{m-d}-p^{m-2d}-\mu p^{s+3d}+\mu p^s+p^{3d})(p^m-1)/(p^d+1)^2$\\&$(p^{2d}-1)$      \\
$  -\mu \zeta^j_p p^{s+d}$  &  $M_4=(p^{m-2d-1}+\mu p^{s-d-1})(p^{m+3d}+p^{m+2d}-p^m-p^{m-d}$\\&$-p^{m-2d}-\mu p^{s+3d}+\mu p^s+p^{3d})(p^m-1)/(p^d+1)^2(p^{2d}-1) $  \\
$\mu p^{s+2d}$ &  $ M_5=(p^{s-d}+\mu)(p^{s+d}+p^s-p^{s-2d}-\mu p^d)(p^{m-4d-1}+\mu$\\&$(p-1)p^{m-2d-1})(p^m-1)/(p^d+1)^2(p^{2d}-1)$  \\
$\mu \zeta^j_p p^{s+2d}$ &  $ M_6=(p^{s-d}+\mu)(p^{s+d}+p^s-p^{s-2d}-\mu p^d)(p^{m-4d-1}-\mu $\\&$p^{m-2d-1})(p^m-1)/(p^d+1)^2(p^{2d}-1)$  \\
$-\mu p^{s+3d}$ &  $ M_7=(p^{s-2d}-\mu)(p^{s-d}+\mu)(p^{m-6d-1}-\mu (p-1) p^{m-3d-1})$\\&$(p^m-1)/(p^d+1)(p^{2d}-1)(p^{3d}+1)$  \\
$-\mu \zeta^j_p p^{s+3d}$ &  $ M_8=(p^{s-2d}-\mu)(p^{s-d}+\mu)(p^{m-6d-1}+\mu  p^{m-3d-1})(p^m-1)$\\&$/(p^d+1)(p^{2d}-1)(p^{3d}+1)$  \\
$0$ &  $ M_9=(1-\mu p^{3s-d}-\mu p^{3s-8d}+p^{m-d}+$\\&$\frac{p^{2m}+p^{2m-9d}+\mu p^{3s-3d}-\mu p^{3s-5d}-p^{m-4d}-p^{m-6d}}{p^d+1})(p^m-1) $  \\
$p^m$ &  $ 1 $  \\
\noalign{\smallskip}
\hline
\end{tabular}
\end{center}
\end{table}



\section*{ACKNOWLEDGMENTS}
The authors would like to thank the anonymous referees for their helpful comments and suggestions, which have greatly improved the presentation and quality of this paper. The research of X. Du was supported by NSFC No. 61772022. The research of C. Fan was supported by NSFC No. 11971395.



\end{document}